\documentclass[12pt,reqno]{amsart}
\usepackage{amsfonts}
\usepackage{amssymb}
\usepackage{color}
\usepackage{dsfont}
\numberwithin{equation}{section}
\theoremstyle{plain}
 \newtheorem{theorem}{Theorem}[section]
 
 \newtheorem{corollary}[theorem]{Corollary}
 \newtheorem{proposition}[theorem]{Proposition}

\theoremstyle{definition}
 \newtheorem{definition}[theorem]{Definition}
 \newtheorem{example}[theorem]{Example}
 \newtheorem{remark}[theorem]{Remark}

\newcommand{\Real}{\mathrm{Re}\,}
\newcommand{\I}{\mathrm{Im}\,}

\newcommand{\bN}{\mathbb{N}}
\newcommand{\bZ}{\mathbb{Z}}
\newcommand{\bR}{\mathbb{R}}

\newcommand{\bC}{\mathbb{C}}

\newcommand{\supp}{\mbox{\rm supp }}

\newcommand{\ind}{\textrm{ind}}
\newcommand{\rom}[1]{$\underline{\overline{\text{#1}}}$}

\newcommand{\one}{\mathbf{1}}

\allowdisplaybreaks

\begin{document}

\vspace{5mm}
\begin{center}
{\bf
{\large
On Quasi-Infinitely Divisible Distributions with a Point Mass}}

\vspace{5mm}

David Berger\\
\end{center}

\vspace{5mm}
An infinitely divisible distribution on $\mathbb{R}$ is a probability measure $\mu$ such that the characteristic function  $\hat{\mu}$ has   a L\'{e}vy-Khintchine representation with characteristic triplet $(a,\gamma, \nu)$, where $\nu$ is a L\'{e}vy measure, $\gamma\in\mathbb{R}$ and $a\ge 0$. A natural extension of such distributions are quasi-infinitely distributions. Instead of a L\'{e}vy measure, we assume that $\nu$ is a "signed L\'{e}vy measure", for further information on the definition see [\ref{Lindner}]. We show that a distribution $\mu=p\delta_{x_0}+(1-p)\mu_{ac}$ with $p>0$ and $x_0\in\bR$, where $\mu_{ac}$ is the absolutely continuous part, is quasi-infinitely divisible if and only if $\hat{\mu}(z)\neq0$ for every $z\in\mathbb{R}$. 
We apply this to show that certain variance mixtures of mean zero normal distributions are quasi-infinitely divisible distributions, and we give an example of a quasi-infinitely divisible distribution that is not continuous but has infinite quasi-L\'{e}vy measure.
Furthermore, it is shown that replacing the signed L\'{e}vy measure by a seemingly more general complex L\'{e}vy measure does not lead to new distributions.  Last but not least it is proven that the class of quasi-infinitely divisible distributions is not open, but path-connected in the space of probability measures with the Prokhorov metric.

\section{Introduction}
The class of infinitely divisible distributions is an important class of distributions, since they correspond in a natural way to L\'{e}vy processes. It is well known that infinitely divisible distributions are characterized by the L\'{e}vy-Khintchine formula in the sense that a distribution $\mu$ is infinitely divisible if and only if there exist $a\ge 0$, $\gamma \in\bR$ and a L\'{e}vy measure $\nu$ such that
\begin{align}\label{lkr}
\hat{\mu}(z)=\exp\left(-\frac{1}{2}az^2+i\gamma z+\int\limits_{\bR} \left(e^{ixz}-1-ixz\one_{[-1,1]}(x)\right)\nu(dx)\right)
\end{align}
for each $z$, where $\hat{\mu}$ denotes the characteristic function of $\mu$.\\
The class of quasi-infinitely divisible distributions generalizes the class of infinitely divisible distributions. By definition, a probability distribution $\mu$ is quasi-infinitely divisible if and only if its characteristic function admits a L\'{e}vy-Khintchine representation (\ref{lkr}), but with $a,\gamma \in\bR$ and $\nu$ being a quasi-L\'{e}vy measure, meaning informally that $\nu$ is a "signed L\'{e}vy measure". See [\ref{Lindner}] and Section 2 below for the precise defintions.
It is easily seen that a distribution $\mu$ is quasi-infinitely divisible if and only if its characteristic function is the quotient of the characteristic functions of two infinitely divisible distributions, equivalently if there exist two infinitely divisible distributions $\mu_1,\mu_2$ such that $\mu_1\ast \mu=\mu_2$. Hence, quasi-infinitely divisible distributions appear naturally in the study of factorisation of infinitely divisible distributions.\\
Applications of quasi-infinitely divisible distributions can be found in physics, see [\ref{Demni1} and \ref{Demni2}], and in insurance mathematics, see [\ref{Zhang}].\\
Although examples of quasi-infinitely divisible distributions have appeared before in the literature (e.g. [\ref{Cuppens2}] and  [\ref{Linnik2}]), a first step to a systematic treatment of these distributions has only been given recently by Lindner et al. [\ref{Lindner}]. They showed in particular that the class of quasi-infinitely divisible distributions is dense in the class of probability distributions with respect to weak convergence, and using the Wiener-L\'{e}vy theorem they showed that a discrete distribution $\mu$ concentrated on a lattice of the form $h\bZ +r$ with $h>0$, $r\in\bR$, is quasi-infinitely divisible if and only if its characteristic function $\hat{\mu}$ does not have zeroes on the real line. They also gave an example of a distribution whose characteristic function has no real zeroes, but such that the distribution nevertheless was not quasi-infinitely divisible. They also studied various distributional properties of quasi-infinitely divisible distributions in terms of the characteristic triplet. Another important class of quasi-infinitely divsible distributions was established  much earlier by Cuppens [\ref{Cuppens1}, Proposition 1; \ref{Cuppens2} Theorem 4.3.7]. He showed that any probability distribution $\mu$ that has an atom of mass greater than $1/2$ is quasi-infinitely divisible.\\
The goal of this paper is to obtain a further class of quasi-infinitely divisible distributions. A main result in this direction will be that a distribution $\mu$ of the form $\mu(dx)=p\delta_{x_0}(dx)+(1-p)f(x)\lambda(dx)$, where $p\in (0,1]$, $x_0\in\bR$, and $f$ being a Lebesgue density, is quasi-infinitely divisible if and only if  its characteristic function has no zeroes on the real line. This can then be seen on the one hand as a counter part to the above mentioned result by Cuppens, and on the other to the above mentioned result by Lindner et al. Its proof makes use of a Wiener-L\'{e}vy theorem due to Krein [\ref{Krein}] for a specific Banach algebra. As a byproduct of our result, we find a quasi-infinitely divisible distribution that is not continuous but has infinite quasi-L\'{e}vy measure, thus answering an open question in [\ref{Lindner}, Open Question 7.2] in the negative. We also show that convex combinations of $N(0,a_i)$-distributions, or more generally variance mixtures of mean zero normal distributions are quasi-infinitely divisible, provided that the lower endpoint $t_1$ of the mixing distribution $\varrho$ is strictly positive and that $\varrho(\{t_1\})>0$. We also treat quasi-infinite divisibility for distributions whose singular part $\mu_d$ is supported on a lattice and such that $\hat{\mu_d}$ has no zeroes on the real line.\\
The results mentioned above can be found in Section 4. In Section 2, we recall basic notation and the formal definition of quasi-infinitely divisible distributions. In Section 3, we address the question if it also makes sense to look at probability distributions $\mu$ whose characteristic function have a L\'{e}vy-Khintchine type representation with a "complex L\'{e}vy measure", and show that this does not lead to a new class, i.e. that no probability distribution exists such that the L\'{e}vy measure in the L\'{e}vy-Khintchine type representation of its characteristic function has a non-zero imaginary part. This result will then be used intensively in the proofs for Section 4. Finally, in Section 5 we show that the complement of the class of quasi-infinitely divisible distributions is also dense with repect to weak convergence, and that the set of quasi-infinitely divisible distributions is path-connected with respect to the Prokhorov topology. This sheds some further light on the topological properties of this class of distributions.
\newline
\section{Notation and Preliminaries}
To fix notation, by a distribution on $\bR$ we mean a probability measure on $(\bR,\mathcal{B})$ with $\mathcal{B}$ being the Borel $\sigma-$algebra on $\bR$, and similarly, by a signed measure on $\bR$ we mean it to be defined on $(\bR,\mathcal{B})$. By a measure on $\bR$ we always mean a positive measure on $(\bR,\mathcal{B})$, i.e. an $[0,\infty]$-valued $\sigma-$additive set function on $\mathcal{B}$ that assigns the value $0$ to the empty set. The Dirac measure at a point $b\in \bR$ will be denoted by $\delta_{b}$, the Gaussian distribution with mean $a\in\bR$ and variance $b\ge 0$ by $N(a,b)$ and the Lebesgue measure by $\lambda(dx)$. Weak convergence of measures will be denoted by "$\stackrel{d}{\to}$" and the Fourier transform at $z\in\bR$ of a finite complex measure $\mu$ on $\bR$ will be denoted by $\hat{\mu}(z)=\int_\bR e^{ixz}\,\mu(dx)$. The convolution of two complex measures $\mu_1$ and $\mu_2$ on $\bR$ is defined by $\mu_1\ast \mu_2(B) =\int_\bR \mu_1(B-x)\,\mu_2(dx)$, $B\in\mathcal{B}$, where $B-x=\{y-x|\,y\in B\}$. The law of a random variable $X$ will be dentoted by $\mathcal{L}(X)$. The real and imaginary part of a complex number $z$ will be denoted by $\Real z$  and  $\I z$, respectively, the imaginary unit will be denoted by $i$. We write $\bN=\{1,2,\dotso\}$, $\bN_0=\bN\cup \{0\}$ and $\bZ,\,\bR,\,\bC$ for the set of integers, real numbers and complex numbers, respectively. The indicator function of a set $A\subset \bR$ is denoted by $\one_A$. By $L^1(\bR, A)$ for $A\subset \bC$ we denote the set of all Borel-measurable functions $f:\bR \to A$ such that $\int_\bR |f(x)|\,\lambda(dx)<\infty$.

Informally, a quasi-L\'{e}vy type measure is the difference of two L\'{e}vy measures. This however is not always a signed measure, which is why the definition is slightly different. Let us recall the following definitions of [\ref{Lindner}]:
\begin{definition}\label{qm}\mbox{}\\
{\it a)} Let $\mathcal{B}_r:=\{B\in \mathcal{B}|\, B\cap (-r,r)=\emptyset\}$ for $r>0$ and $\mathcal{B}_0:=\cup_{r>0}\mathcal{B}_r$ be the class of all Borel sets that are bounded away from zero. Let $\nu:\mathcal{B}_0\to\bR$ be a function such that $\nu|_{\mathcal{B}_r}$ is a finite signed measure for each $r>0$ and denote the total variation, positive and negative part of $\nu_{|\mathcal{B}_R}$ by $|\nu_{|\mathcal{B}_r}|,\nu^{+}_{|\mathcal{B}_r}$ and $\nu^{-} _{|\mathcal{B}_r}$, respectively. Then the total variation $|\nu|$, the positive part $\nu^{+}$  and the negative part $\nu^{-}$ of $\nu$ are defined to be the unique measures on $(\bR,\mathcal{B})$ satisfying
\begin{align*}
|\nu|(\{0\})=\nu^{+}(\{0\})=\nu^{-}(\{0\})=0
\end{align*}
and
\begin{align*}
|\nu|(A)=|\nu_{|\mathcal{B}_r}|(A),\,\nu^{+}(A)=\nu^{+}_{|\mathcal{B}_r}(A),\,\nu^{-}(A)=\nu^{-} _{|\mathcal{B}_r}(A)
\end{align*}
when $A\in \mathcal{B}_r$ for some $r>0$.\\ 
{\it b)} A \emph{quasi-L\'{e}vy type measure} is a function satisfying the condition of {\it a)} such that its total variation $|\nu|$ satisfies $\int_\bR (1\wedge x^2 )|\nu|(dx)<\infty.$
\end{definition}
\begin{definition}\label{def1}{\it a)} Let $\mu$ be a probability distribution on $\bR$. We say that $\mu$ is \emph{quasi-infinitely divisible} if its characteristic function has a representation
\begin{align*}
\hat{\mu}(z)=\exp\left( -\frac12 az^2+ i\gamma z+\int\limits_{\bR}(e^{ixz}-1-ixz\one_{[-1,1]}(x))\nu(dx) \right),
\end{align*}
where $a,\gamma \in\bR$ and $\nu$ a quasi-L\'{e}vy-type measure. The \emph{characteristic triplet} $(a,\gamma,\nu)$ of $\mu$ is unique (see [\ref{Sato}, Exercise 12.2]), and $a$ is called the \emph{Gaussian variance} of $\mu$.\\
{\it b)} A quasi-L\'{e}vy type measure $\nu$ is called \emph{quasi-L\'{e}vy measure}, if additionally there exist a quasi-infinitely divisible distribution $\mu$  and some $a,\gamma \in\bR$ such that $(a,\gamma,\nu)$ is the characteristic triplet of $\mu$. We call $\nu$ the \emph{quasi-L\'{e}vy measure of $\mu$}.\\
{\it c)} Let $\mu$ be quasi-infinitely divsible with characteristic triplet $(a,\gamma,\nu)$. If $\int_{[-1,1]} |z| \,|\nu|(dx)<\infty$,  then we call $\gamma_0:= \gamma-\int_{[-1,1]} z\,\nu(dz)$ the \emph{drift} of $\mu$. In that case, the characteristic function of $\mu$ allows the representation
\begin{align*}
\hat{\mu}(z)=\exp\left(-\frac{1}{2}az^2+i\gamma_0 z+\int\limits_{\bR} (e^{ixz}-1)\nu(dx)\right).
\end{align*}
\end{definition}

\section{Complex quasi-L\'{e}vy Type measures}
As stated in Definition \ref{def1}, a quasi-infinitely divisible distribution is a probability distribution $\mu$ whose characteristic function admits a L\'{e}vy-Khintchine type representation with a quasi-L\'{e}vy-type measure. It is natural to ask if there are further distributions whose characteristic function allows a L\'{e}vy-Khintchine type representation with a complex quasi-L\'{e}vy-type measure. Theorem \ref{theorem1} below shows that this is not the case.  But before that we need a precise definition:
\begin{definition}\label{cqltm}\mbox{}\\
A \emph{complex quasi-L\'{e}vy type measure} is a function $\nu:\mathcal{B}_0\to\bC$ such that $\Real \nu$ and $\I \nu$ are quasi-L\'{e}vy type measures.
\end{definition}
The integral of a function $f:\bR\to\bC$ satisfying $|f(x)|\le C(x^2\wedge 1)$ for some constant $C$ with respect to a complex quasi-L\'{e}vy type measure can be defined in the obvious way as
\begin{align*}
\int\limits_{\bR}f(x)\nu(dx):=\lim\limits_{r \downarrow 0} \int\limits_{|x|\ge r} f(x)\nu_{|\mathcal{B}_r}(dx)=\int\limits_{\bR}f(x)(\Real \nu)(dx)+i\int\limits_{\bR}f(x)(\I \nu)(dx),
\end{align*}
which shows in particular that $x \mapsto e^{izx}-1-izx\one_{[-1,1]}(x)$ is integrable for every $z\in\bR$.\\
We now come to the aforementioned result:
\begin{theorem}\label{theorem1}
Let $\mu$ be a distribution on $\bR$ whose characteristic function allows a representation of the form
\begin{align}\label{clkr}
\hat{\mu}(z)=\exp\left( i\gamma z- \frac12 az^2+\int\limits_{\bR}(e^{ixz}-1-izx\one_{[-1,1]}(x))\nu(dx) \right)
\end{align}
where $\gamma\in \bC$, $a\in\bC$ and $\nu$ is a complex quasi-L\'{e}vy-type measure. Then $a,\gamma \in\bR$, $a\ge 0$ and $\I \nu=0$, i.e. $\nu$ is a quasi-L\'{e}vy measure and $\mu$ is quasi-infinitely divisible.
\end{theorem}
\begin{proof}
We see that
\begin{align*}
|\hat{\mu}(z)|^2=\hat{\mu}(z)\hat{\mu}(-z)&=\exp\left( -az^2+\int\limits_{\bR}(e^{ixz}+e^{-ixz}-2)\nu(dx) \right)\\
&=\exp\left( -az^2+\int\limits_{\bR}2(\cos(xz)-1)\nu(dx)\right)
\end{align*}
and
\begin{align*}
\frac{\hat{\mu}(z)}{\hat{\mu}(-z)}=\exp\left( i2\gamma z+i\int\limits_{\bR} 2(\sin(xz)-zx\one_{[-1,1]}(x))\nu(dx) \right).
\end{align*}
As $|\hat{\mu}(z)|>0$ for every $z\in\bR$, we see that 
\begin{align*}
|\hat{\mu}(z)|^2=\exp\left( \log (|\hat{\mu}(z)|^2) \right),
\end{align*}
where $\log$ is the natural logarithm. As the distinguished logarithm is uniquely determined (see [\ref{Sato}], Lemma 7.6) and
\begin{align*}
 g(z)=-az^2+\int\limits_{\bR}2(\cos(xz)-1)\nu(dx)
\end{align*}
is continuous and $g(0)=0$, we conclude that
\begin{align*}
-\frac{1}{2}az^2+\int\limits_{\bR} (\cos(xz)-1)\nu(dx) \in\bR
\end{align*}
for every $z\in\bR$ and especially we obtain
\begin{align*}
-\frac{1}{2}\I az^2 + \int\limits_{\bR}(\cos(xz)-1)\nu^{\I}(dx)=0.
\end{align*}
Furthermore, as $|\hat{\mu}(z)|=|\overline{\hat{\mu}(z)}|=|\hat{\mu}(-z)|$, we conclude that
\begin{align*}
\gamma z+\int\limits_{\bR} (\sin(xz)-zx\one_{[-1,1]}(x))\nu(dx) \in\bR
\end{align*}
for every $z\in\bR$. It follows that
\begin{align*}
0=\I \gamma z+\int\limits_{\bR} (\sin(xz)-zx\one_{[-1,1]}(x))\nu^{\I}(dx)
\end{align*}
for every $z\in\bR$. At last, with the quasi-L\'{e}vy measure $\nu^{\I}$ we obtain a L\'{e}vy-Khintchine formula for $\delta_0(dx)$, because
\begin{align*} 
0&=i \I\gamma z-\frac{1}{2}\I a z^2+\int\limits_{\bR}(\cos(xz)-1)\nu^{\I}(dx)+i\int\limits_{\bR}(\sin(xz)-xz \one_{[-1,1]})\nu^{\I}(dx)\\
&=i \I\gamma z-\frac{1}{2}\I a z^2+\int\limits_{\bR}(e^{ixz}-1-ixz\one_{[-1,1]}(x))\nu^{\I}(dx).
\end{align*}
Hence
\begin{align*}
\hat{\delta}_0(z)=1=\exp(0)=\exp\left(i \I\gamma z-\frac{1}{2}\I a z^2+\int\limits_{\bR}(e^{ixz}-1-ixz\one_{[-1,1]}(x))\nu^{\I}(dx) \right).
\end{align*}
The uniqueness of the characteristic triplet of quasi-infinitely divisible distributions then shows that $\I \gamma=\I a=0$ and that $\nu^{\I}$ is the null-measure. Hence $\mu$ is quasi-infinitely divisible. That $a \ge 0$ follows from Lemma 2.7 in  [\ref{Lindner}].
\end{proof}
\begin{remark}
Theorem \ref{theorem1} is very helpful to prove quasi-infinite divisibility of certain distributions, as it is often easier to establish a L\'{e}vy-Khintchine type representation with a complex quasi-L\'{e}vy-type measure rather then directly with a quasi-L\'{e}vy-type measure. An example is the proof of Theorem 8.1 in [\ref{Lindner}]. There it is shown, using the L\'{e}vy-Wiener-Theorem, that a distribution $\mu$ on $\bZ$ with $\hat{\mu}(z)\neq 0$ for all $z$ allows a L\'{e}vy-Khintchine type representation with a complex L\'{e}vy-type measure $\nu=\sum_{k\in\bZ\setminus \{0\}}b_k\delta_k$ for some summable sequence $b_k\in\bC$. There, it is shown using an involved approximation argument that the $b_k$ are actually real and hence $\mu$ quasi-infinitely divisible. This step can now be simplified considerably by using Theorem \ref{theorem1}. 
\end{remark}
\section{Some new quasi-infinitely divisible distributions}
\subsection{Absolutely continuous distributions plus a pointmass}
In this section we will look at distributions of the form 
\begin{align}\label{1}
\mu(dx)=p\delta_{x_0}(dx)+(1-p)f(x)\lambda(dx),
\end{align}
where $\lambda$ is the Lebesgue measure, $f$ a Lebesgue density, $x_0\in\bR$ and $p\in (0,1)$. We first specialize to $x_0=0$. The characteristic function is then given by
\begin{align*}
\hat{\mu}(z)=p+(1-p) \hat{f}(z),
\end{align*}
where $\hat{f}(z)=\int_\bR e^{ixz} f(x)\,\lambda(dx)$.
We want to use a similar argument as in $[\ref{Lindner}]$ in order to show every distribution $\mu$ of the form $(\ref{1})$ is quasi-infinitely divisible if and only if $\hat{\mu}(z)\neq0$ for every $z\in\bR$. We denote by $\overline{\bR}$ the compactification of $\bR$, i.e. $\overline{\bR}=\bR\cup \{-\infty,\infty\}$. A characteristic function $\hat{\mu}$ of the distribution $\mu$ of the form (\ref{1}) is then nonzero on the set $\overline{\bR}$ if and only if it is nonzero on $\bR$. This follows directly from the Riemann-Lebesgue Lemma.\\
At first let us fix some notations.
\begin{definition} We denote by $W(\bR,\bC)$ the space
\begin{align*}
W(\bR,\bC):=\{F:\bR\to\bC\,| \exists p\in\bC, f\in L^1(\bR,\bC) \textrm{ such that } F(z)=p+\int\limits_{\bR} f(x)e^{ixz}\,\lambda(dx)\textrm{ for all }z\in\bR \}.
\end{align*}
With the norm $||F||=|p|+||f||_{L^1(\bR,\bC)}$ for $F(z)=p+\int\limits_{\bR} f(x)e^{ixz}\,\lambda(dx)$ the normed space $(W(\bR,\bC),||\cdot||)$ becomes a Banach algebra (see [\ref{Liflyand}, Theorem 4.1]).
\end{definition}
\begin{definition}
Let $W(\bR,\bC)\ni F(x)=p+\int_{\bR}e^{ixz}f(z)\lambda(dz)\neq 0$ for every $x\in \bR$ and $p\in\bC\setminus\{0\}$. Then we can interpret $F$ as a closed curve in $\bC$. By the property of the distinguished logarithm, there exists a continuous function $g:\bR \to \bR$ such that
\begin{align*}
\frac{F(x)}{|F(x)|}=\exp\left( ig(x) \right)\quad\textrm{ for all }x\in\bR.
\end{align*}
Then the \emph{index} $\ind(F)$ of $F$ is defined as
\begin{align*}
\ind(F):=\frac{1}{2\pi }(\lim\limits_{z\to +\infty}g(z)-\lim\limits_{z\to -\infty}g(z))=:\frac{1}{2\pi}(g(\infty)-g(-\infty)).
\end{align*}
\end{definition}
\begin{remark}
By the Riemann-Lebesgue, it is relatively easy to see that $\ind(F)$ is well-defined and $\ind(F)\in\bZ$. Also, for $F(x)=p+\int_{\bR}e^{ixz}f(x)\,\lambda(dx)\in W(\bR,\bC)$ we have $F(x)\neq 0$ for all $x\in \overline{\bR}$ if and only if $p\neq 0$ and $F(x)\neq 0$ for all $x\in\bR$.
\end{remark}
The following is now the key ingredient for identifying further quasi-infinitely divisible distributions:
\begin{theorem} [\ref{Krein}, Theorem L, p.175]\label{2}
Let $F(x)=p+\int\limits_{\bR}e^{ixz}f(z)\lambda(dz)$ for a function $f\in L^1(\bR,\bC)$ and some $p\in \bC\setminus\{0\}$. Furthermore, assume that $F(x)\neq 0$ for every $x\in\bR$ and {\normalfont ind}$(F)=0$.
Then there exist some $q\in\bC$ and a function $g\in L^1(\bR,\bC)$ such that 
\begin{align*}
F(x)=\exp\left(q+\int\limits_{\bR} e^{ixz}g(z)\,\lambda(dz) \right)\quad\textrm{for all }x\in\bR.
\end{align*}
\end{theorem}
With the aid of Theorem \ref{2} we can now give a L\'{e}vy-Khintchine type representation for arbitrary $F\in W(\bR,\bC)$ that do not vanish on $\overline{\bR}$ and are such that $F(0)=1$.
\begin{theorem}\label{Theorem2}\mbox{}\\
Let $F\in W(\bR,\bC)$ with $F(z)\neq 0$ for every $z\in\overline{\bR}$ and $F(0)=1$. Denote by $m$ the index of $F$. Then there is some function $g\in L^1(\bR,\bC)$ such that
\begin{align}\label{special}
F(z)=\exp\left( \int\limits_{\bR}(e^{ixz}-1)\left(g(x)+\frac{me^{-|x|}}{|x|}(\one_{(0,\infty)}(x)-\one_{(-\infty,0)}(x))\right)\lambda(dx) \right)
\end{align}
for all $z\in\bR.$
\end{theorem}
\begin{proof}
{\it a)} Let us first assume that $m=\ind(F)=0$. By Theorem \ref{2} there exist a constant $c\in\bC$ and a function $g\in L^1(\bR,\bC)$ such that
\begin{align*}
F(z)=\exp\left( c+\int\limits_{\bR}e^{izx}g(x)\,\lambda(dx) \right).
\end{align*}
As $F(0)=1$, we conclude that
\begin{align*}
c+\int\limits_{\bR}g(x)\,\lambda(dx)\in 2\pi i\bZ.
\end{align*}
Hence, we can write 
\begin{align*}
F(z)&=\exp\left(c+\int\limits_{\bR}g(x)\,\lambda(dx)\right)\exp\left(\int\limits_{\bR}(e^{ixz}-1)g(x)\,\lambda(dx)\right)\\
&=\exp\left(\int\limits_{\bR}(e^{ixz}-1)g(x)\,\lambda(dx)\right).
\end{align*}
{\it b)} Now assume that $0\neq m=\ind(F)\in\bN$. Define the function $Q:\bR\to \bC$ by
\begin{align*}
Q(z)=\frac{(z-i)^m}{(z+i)^m},\quad z\in\bR.
\end{align*}
Since
\begin{align*}
\frac{z+i}{z-i}=1+\frac{2i}{z-i}=1-\frac{2}{1+iz}=1-2\int\limits_{-\infty}^0 e^{x}e^{ixz}\,\lambda(dx),\quad z\in\bR,
\end{align*}
it follows that $z\mapsto \frac{z+i}{z-i}\in W(\bR,\bC)$ and hence, since $W(\bR,\bC)$ is a Banach algebra, that also $Q^{-1}\in W(\bR,\bC)$ and that $Q^{-1}F\in W(\bR,\bC)$. Then obviously $Q^{-1}(z)F(z)\neq 0$ for all $z\in\overline{\bR}$, and by the proof of Theorem 2.2, p. 180 in Krein [\ref{Krein}], it follows that $\ind(Q^{-1}F)=0$ and hence $\ind(Q(0)Q^{-1}F)=0$. Hence, by part {\it a)} already proved, there is some $g\in L^1(\bR,\bC)$ such that
\begin{align}\label{12}
Q(0)Q^{-1}(z)F(z)=\exp\left( \int\limits_{\bR} (e^{ixz}-1)g(x)\,\lambda(dx) \right),\quad \textrm{ for all }z\in\bR.
\end{align}
But
\begin{align*}
\frac{1}{(z+i)^m}=(-i)^m \frac{1}{(1-iz)^m}=(-i)^m \exp\left( \int_{0}^\infty (e^{ixz}-1)\frac{me^{-x}}{x}\lambda(dx) \right)\quad \textrm{for all }z\in\bR
\end{align*}
(see [\ref{Sato}], Example 8.10), hence
\begin{align*}
(z-i)^m =\left((-1)^m \frac{1}{(-z+i)^m}\right)^{-1}=(-1)^mi^m\exp\left(-\int\limits_{0}^\infty (e^{-izx}-1)m \frac{e^{-x}}{x}\,\lambda(dx) \right)
\end{align*}
so that
\begin{align*}
Q(z)=\left(\frac{z-i}{z+i}\right)^m=(-1)^m \exp\left(\int\limits_{\bR} (e^{ixz}-1)\left( \frac{me^{-x}}{x}\one_{(0,\infty)}(x)-\frac{me^{-|x|}}{|x|}\one_{(-\infty,0)}(x) \right) \lambda(dx)\right).
\end{align*}
Observe that $(-1)^m=Q(0)$. Together with (\ref{12}) this gives the desired result when $m\in\bN$.\\
{\it c)}  Now assume that $m=\ind(F)\in -\bN$. Then $x\mapsto F(-x)=:G(x)\in W(\bR,\bC)$ with $\ind (G)=-m$. The result then follows from {\it b)}.
\end{proof}

Similarly as in Lindner et al. [\ref{Lindner}], who showed that a distribution on $\bZ$ is quasi-infinitely divisible if and only if its characteristic function has no zeroes, we can now prove that a distribution whose singular part consists of a non-trivial atom is quasi-infinitely divsible if and only if its characteristic function has no zeroes:
\begin{theorem}\label{cor1}
Let $\mu$ be a probability distribution of the form $(\ref{1})$. Then $\mu$ is quasi-infinitely divisible if and only if $\hat{\mu}(z)\neq 0$ for every $z\in\bR$. In that case, the quasi-L\'{e}vy measure $\nu$ of $\mu$ is given by
\begin{align*}
\left(g(x)+\frac{me^{-|x|}}{|x|}(\one_{(0,\infty)}(x)-\one_{(-\infty,0)}(x))\right)\lambda(dx),
\end{align*}
where $g\in L^1(\bR,\bR)$ and $m$ is the index of $\widehat{\mu\ast \delta_{-x_0}}$. Furthermore, $\int_{-1}^1 |x| |\nu| (dx)<\infty$, $\mu$ has drift $x_0$ and Gaussian variance $0$. Finally, $|\nu|$ is finite if and only if $m=0$, and if $m\neq 0$, then $\nu^{-}(\bR)=\nu^{+}(\bR)=\infty$. 
\end{theorem}
\begin{proof}
That $\hat{\mu}(z)\neq 0$ for all $z\in\bR$ is obviously necessary for $\mu$ to be quasi-infinitely divisible. To see that it is sufficient, it is sufficient to assume that $x_0=0$, since $\mu$ is quasi-infinitely divsible if and only if $\tilde{\mu}:=\mu\ast \delta_{-x_0}$ is quasi-infinitely divsible. By Theorem \ref{Theorem2} we see that $\tilde{\mu}$ has a L\'{e}vy-Khintchine representation given by
\begin{align*}
\hat{\tilde{\mu}}(z)=\exp\left( \int\limits_{\bR}(e^{ixz}-1)(g(x)+\frac{me^{-|x|}}{|x|}(\one_{(0,\infty)}(x)-\one_{(-\infty,0)}(x)))\lambda(dx) \right)
\end{align*} 
for all $z\in \bR$ with some $g\in L^1(\bR,\bC)$. Then Theorem \ref{theorem1} implies that $g\in L^1(\bR,\bR)$ and $\mu$ is quasi-infinitely divisible. The remaining assertions are clear.
\end{proof}
\begin{example}\label{ex1}
It is worth noting that distributions of the above form with non-zero index can indeed occur. For example, consider the distribution
\begin{align*}
\mu(dx)=\frac{1}{1000}\delta_0(dx) + \frac{999}{1000}\rho(dx),
\end{align*}
where $\rho=N(1,1)$. Then $\hat{\mu}(z)=\frac{1}{1000}+\frac{999}{1000}e^{iz}e^{-z^2/2}\neq 0$ for all $z\in\bR$. Observing that $\hat{\mu}(\pi)<0$, $\Real\hat{\mu}(z)>0$ for all $z\ge 2\pi$, $\I \hat{\mu}(z)=e^{-z^2/2}\sin(z)$ it is easy to see that $\hat{\mu}$ has index $2$. Hence $\mu$ is quasi-infinitely divisible with quasi-L\'{e}vy measure $\nu$ satisfying $\nu^{-}(\bR)=\nu^{+}(\bR)=|\nu|(\bR)=\infty$ by Theorem \ref{cor1}.
\end{example}
\begin{remark} In [\ref{Lindner}, Open Question 7.2] it was asked that if for a quasi-infinitely divisible distribution $\mu$ with triplet $(a,\gamma,\nu)$ continuity of $\mu$ is equivalent to $a\neq 0$ or $|\nu|(\bR)=\infty$. Example \ref{ex1} answers this question in the negative. Indeed, the distribution $\mu$ there is not continuous, but the total variation of the quasi-L\'{e}vy measure is infinite.
\end{remark}
\begin{remark}
It is known that distributions of the form $(\ref{1})$ with $p\in [0,1]$, $x_0=0$ and $f$ vanishing on $(-\infty,0)$ are infinitely divisible provided $\log f$ is convex on $(0,\infty)$ or $f$ is completely montone on $(0,\infty)$, see [\ref{Sato}, Theorem 51.4, 51.6]. Theorem \ref{cor1} shows that quasi-infinite divsibility can be achieved for a much wider class of distributions $\mu$ of this type, provided $p>0$ and $\hat{\mu}$ has no zeroes, but with no other assumptions on $f$.
\end{remark}
As in [\ref{Lindner}, Corollary 8.3], one can now see that factors of a quasi-infinitely divisible distribution of the form (\ref{1}) are also quasi-infinitely divisible. 
\begin{corollary}\label{fac1}
Let $\mu$ be a distribution of the form $(\ref{1})$ and $\mu_1,\mu_2$ probability distributions such that $\mu=\mu_1\ast \mu_2$. Then $\mu$ is quasi-infinitely divisible if and only if $\mu_1$ and $\mu_2$ are quasi-infinitely divisible.
\end{corollary}
\begin{proof}
We write
\begin{align*}
\mu_i=\mu_i^{d}+\mu_i^{ac}+\mu_i^{cs}
\end{align*}
for $i=1,2$ where $\mu_i^{d}$ is the discrete part, $\mu_i^{ac}$ is the absolute continuous part and $\mu_i^{cs}$ is the continuous singular part.
Hence, we can write 
\begin{align*}
\mu=\mu_1^{d}\ast\mu_2^{d}+\mu_1^{d}\ast\mu_2^{cs}+\mu_1^{d}\ast\mu_2^{ac}+\mu_1^{cs}\ast\mu_2^{d}+\mu_1^{cs}\ast\mu_2^{cs}+\mu_1^{cs}\ast\mu_2^{ac}+\mu_1^{ac}\ast\mu_2^{d}+\mu_1^{ac}\ast\mu_2^{cs}+\mu_1^{ac}\ast\mu_2^{ac}.
\end{align*}
As $\mu_1^{d}\ast\mu_2^{d}$ is the only discrete part, we conclude that $\mu_1$ and $\mu_2$ each have exactly one point mass. Moreover, $\mu_i^{cs}$ has to be zero for $i=1,2$, as $\mu_i^{cs}\ast \mu_j^{d}$ is continuous singular for $j\neq i$. So we can write
\begin{align*}
\mu=(p_1\delta_{z_1}(dx)+\mu_1^{ac})\ast(p_2\delta_{z_2}(dx)+\mu_2^{ac}),
\end{align*}
such that $p=p_1p_2$. It follows from Theorem $\ref{cor1}$ that $\mu_1$ and $\mu_2$ are quasi-infinitely divisible if and only if $\mu$ is quasi-infinitely divisible.
\end{proof}
It follows from Theorem \ref{cor1} that if $\mu$ is a distribution of the form (\ref{1}) with $\hat{\mu}(z)\neq 0$ for all $z\in\bR$, and if $\mu'$ is an infinitely-divisible (or quasi-infinitely divisible) distribution, then $\mu'\ast \mu$ is again quasi-infinitely divisible. This observation can be used to derive quasi-infinite divisibility of certain variance mixtures of normal distributions or more generally mixtures of distributions of L\'{e}vy processes. More precisely, we have:
\begin{corollary}\label{cor4.10}\mbox{}\\
{\it a)} Let $\varrho$ be a probability distribution on $\bR$  with $\varrho((-\infty,t_1))=0$ and $\varrho(\{ t_1\})>0$ for some $t_1>0$. Let $L=(L_t)_{t\ge 0}$ be a L\'{e}vy process such that $\mathcal{L}(L_t)$ is absolutely continuous for each $t>0$. Define the mixture $\mu:=\int_{[t_1,\infty)}\mathcal{L}(L_t)\,\varrho(dt)$ by
\begin{align}\label{eq100}
\mu(B):= \int\limits_{[t_1,\infty)} \mathcal{L}(L_t)(B)\,\varrho(dt),\quad B\in \mathcal{B}.
\end{align}
Then $\mu$ is quasi-infinitely divisible if and only if $\hat{\mu}(z)\neq 0$ for all $z\in\bR$. In particular, if $\varrho =\sum_{i=1}^n p_i\delta_{t_i}$ with $t_1<t_2<\dotso<t_n$ and $0<p_1,\dotso,p_n <1$, $\sum_{i=1}^n p_i=1$, then $\mu=\sum_{i=1}^n p_i \mathcal{L}(L_{t_i})$ is quasi-infinitely divisible if and only if $\hat{\mu}(z)\neq 0$ for all $z\in\bR$.\\
{\it b)} The assumption $\hat{\mu}(z)\neq 0$ for all $z\in\bR$ for $\mu$ of the form (\ref{eq100}) is in particular satisfied when $\mathcal{L}(L_1)$ is symmetric.
\end{corollary} 
\begin{proof}
{\it a)} Write $\mu_t=\mathcal{L}(L_t)$. Then
\begin{align*}
\mu&=\int\limits_{[t_1,\infty)} \mu_t \,\varrho(dt)=\varrho(\{t_1\}) \mu_{t_1}+\int\limits_{(t_1,\infty)}\mu_t \varrho(dt)\\
&=\mu_{t_1}\ast\left(\varrho(\{t_1\})\delta_0+\int\limits_{(t_1,\infty)}\mu_{t-t_1}\varrho(dt)\right).
\end{align*}
Assume that $\hat{\mu}(z)\neq 0$ for all $z\in\bR$. Then $(\varrho(\{t_1\})\delta_0+\int\limits_{(t_1,\infty)}\mu_{t-t_1}\varrho(dt))\,\,\widehat{}\,(z)\neq 0$ for all $z\in\bR$. Since $\mu_{t-t_1}$ is absolutely continuous for all $t>t_1$, so is $\int\limits_{(t_1,\infty)}\mu_{t-t_1}\varrho(dt)$. Hence $\varrho(\{t_1\})\delta_0+\int\limits_{(t_1,\infty)}\mu_{t-t_1}\varrho(dt)$ is quasi-infinitely divisible by Theorem \ref{cor1}. Since  $\mu_1$ is infinitely divisible, this shows quasi-infinite divisibility of $\mu$. The converse and the specialization to $\varrho=\sum_{i=1}^n p_i\delta_{t_i}$ are clear.\\
{\it b)} This follows from the fact that
\begin{align*}
\hat{\mu}(z)=\int\limits_{[t_0,\infty)} \hat{\mu}_t \,\varrho(dt)
\end{align*}
and that $\hat{\mu}_t(z)>0$ when $\mu_t$ is symmetric.
\end{proof}
Corollary \ref{cor4.10} applies in particular when $L$ is a standard Brownian motion and hence to variance mixtures of the form $\sum_{i=1}^n p_i N(0,a_i)$ or more generally to variance mixtures of the form $\int_{[t_1,\infty)} N(0,t)\varrho(dt)$ when $\varrho(\{t_1\})>0$ and $t_1>0$. That a variance mixture of the form $pN(0,a)+(1-p)N(0,b)$ with $0<a<b$ and $p\in (1/2,1)$ is quasi-infinitely divisible was already observed in [\ref{Lindner}, Example 3.6]. Corollary \ref{cor4.10} improves in particular on that result in the sense that it shows that $p\in (1/2,1)$ is superfluous.\\
Observe that a distribution of the form $\int_{[0,\infty)}N(0,t)\varrho(dt)$ cannot be infinitely divisible when the support of $\varrho$ is additionally bounded and $\varrho$ is non-degenerate, see [\ref{Kelker}, Theorem 2], but it is infinitely divisible if $\varrho$ is infinitely divisible (e.g.  [\ref{Steutel}, Example \rom{IV}, 11.6). Hence Corollary \ref{cor4.10} sheds some further light onto the behaviour of variance mixtures of normal distributions.
\begin{remark}\label{rem4.11}
 Corollary \ref{cor4.10} continues to hold when $L$ is replaced by an additive process for which all increment distribution $\mathcal{L}(L_t-L_s)$ with $0<s<t$ are absolutely continuous. The proof is exactly the same as in Corollary \ref{cor4.10}. In particular, $\mu=\sum_{i=1}^n p_i N(b_i,a_i)$ is quasi-infinitely divisible for $0<p_1,\dotso,p_n<1,$ $\sum_{i=1}^n p_i=1$, \\$0<a_1<a_2<\dotso<a_n$ and $b_1,\dotso,b_n\in\bR$ if and only if $\hat{\mu}(z)\neq 0$ for all $z\in\bR$.
\end{remark}
\subsection{Absolutely continuous distributions plus a lattice distribution}\mbox{}\\
Until now we considered distributions of the form $p\delta_{x_0}+\mu_{ac}$, where $p>0$ and $\mu_{ac}$ was absolutely continuous. We will now generalise Theorem \ref{cor1} to distributions of the form 
\begin{align*}
\mu=\mu_d+\mu_{ac},
\end{align*}
where $\mu_d$ is a non-zero discrete measure supported on a lattice with non-vanishing  characteristic function, and $\mu_{ac}$ is absolutely continuous. First we recall the classical Wiener-lemma.
\begin{theorem}\label{Wiener}
Let $f(z)=\sum\limits_{k\in\bZ}c_ke^{ikz}$ with $\sum_{k\in\bZ}|c_k|<\infty$ such that $f(z)\neq 0$ for every $z\in\bR$. Then there exists a function $g(z)=\sum\limits_{k\in\bZ}d_k e^{ikz}$ with $\sum_{k\in\bZ}|d_k|<\infty$ such that
\begin{align*}
f(z)g(z)=1
\end{align*}
for every $z\in\bR$.
\end{theorem}
\begin{proof}
See [\ref{Muscalu}, Corollary 4.27].
\end{proof}
Now we prove the aforementioned generalisation.
\begin{theorem}\label{theorem3}
Let $\mu$ be a probability distribution of the form
\begin{align*}
\mu(dx)=\mu_d(dx)+f(x)\lambda(dx),
\end{align*}
where $\mu_d$ is a non-zero discrete measure supported on a lattice of the form $r+h\bZ$ for some $r\in\bR$ and $h>0$, and $f\in L^1(\bR,[0,\infty))$. Then $\mu$ is quasi-infinitely divisible if and only if $\hat{\mu}(z)\neq 0$ for all $z\in\bR$. In that case, the Gaussian variance of $\mu$ is zero and the quasi-L\'{e}vy measure $\nu$ satisfies $\int_{-1}^1 |x|\,|\nu|(dx)<\infty$.
\end{theorem}
\begin{proof}
By shifting and scaling the distribution, we assume without loss of generality that $\supp \mu_d\subset \bZ$, hence we can write
\begin{align*}
\mu_d(dx)=\sum\limits_{k\in\bZ} p_k \delta_k(dx).
\end{align*}
Its characteristic function is given by
\begin{align*}
\hat{\mu}_d(z)=\sum\limits_{k\in\bZ} p_k e^{ikz}.
\end{align*}
Now by Theorem \ref{Wiener} there exists a function $g$ with $g(z)\hat{\mu}_d(z)=1$ for all $z\in\bR$ and
\begin{align*}
g(z)=\sum\limits_{k\in\bZ}c_k e^{ikz}
\end{align*}
and $\sum_{k\in\bZ}|c_k|<\infty$. We can associate a complex measure $\varrho$ such that
\begin{align*}
\varrho(dx)=\sum\limits_{k\in\bZ} c_k \delta_k(dx),
\end{align*}
and especially we conclude that
\begin{align*}
\mu_d\ast \varrho=\delta_0.
\end{align*}
Now we decompose $\mu$ as follows
\begin{align*}
\mu=\mu_d\ast(\delta_0+\varrho \ast \mu_{ac}),
\end{align*}
where $\mu_{ac}(dx)=f(x)\lambda(dx)$. Since $\hat{\mu}_d(z)\neq 0$ for all $z\in\bR$, $\mu_d$ is quasi-infinitely divisible with finite quasi-L\'{e}vy measure by [\ref{Lindner}, Theorem 8.1]. Furthermore, $\varrho\ast \mu_{ac}$ is absolutely continuous, hence there exists some $g\in L^1(\bR,\bC)$ such that $\varrho\ast\mu_{ac}=g(x)\lambda(dx)$. Theorem \ref{Theorem2} then shows that $(\delta_0+\varrho\ast \mu_{ac})\,\,\widehat{}\,$ has a (possibly complex) L\'{e}vy-Khintchine type representation, and since $\mu_{d}$ is quasi-infinitely divisible, it follows from Theorem \ref{theorem1} that $\mu$ is quasi-infinitely divisible if and only if $\hat{\mu}(z)\neq 0$ for all $z\in\bR$. That the Gaussian variance of $\mu$ is zero and the quasi-L\'{e}vy measure satisfies  $\int_{-1}^1 |x|\,|\nu|(dx)<\infty$ then follows from Theorem \ref{cor1} and its proof.
\end{proof}
As in Corollary \ref{fac1}, we can now show that factors of a quasi-infinitely divisible distribution of the form $\mu=\mu_d+\mu_{ac}$ as above are also quasi-infinitely divisible. 
\begin{corollary} Let $\mu$ be of the form $\mu(dx)=\mu_d(dx)+f(x)\lambda(dx)$ as above with $\hat{\mu}_d(z)\neq 0$ for all $z\in\bR$ and $\mu_d$ being concentrated on a lattice. Let $\mu=\mu_1\ast\mu_2$ be a factorisation of $\mu$. Then $\mu$ is quasi-infinitely divisible if and only if $\mu_1$ and $\mu_2$ are quasi-infinitely divisible.
\end{corollary}
\begin{proof}
Assume $\hat{\mu}(z)\neq 0$ for all $z\in\bR$. As in Corollary \ref{fac1} we can write $\mu$ as
\begin{align*}
\mu=\mu_1^{d}\ast\mu_2^{d}+\mu_1^{d}\ast\mu_2^{cs}+\mu_1^{d}\ast\mu_2^{ac}+\mu_1^{cs}\ast\mu_2^{d}+\mu_1^{cs}\ast\mu_2^{cs}+\mu_1^{cs}\ast\mu_2^{ac}+\mu_1^{ac}\ast\mu_2^{d}+\mu_1^{ac}\ast\mu_2^{cs}+\mu_1^{ac}\ast\mu_2^{ac}.
\end{align*}
Now we know that $\mu_1^{d}\ast\mu_2^{d}$ is the only discrete part of $\mu$, so we conclude
\begin{align*}
\mu_d=\mu_1^{d}\ast\mu_2^{d}.
\end{align*}
By [\ref{Lindner}, Corollary 8.3] we know that $\mu_1^{d}$ and $\mu_2^{d}$ are lattice distributions with non-vanishing characteristic functions. $\mu^d_i\ast\mu^{cs}_j$ is continuous singular for $i\neq j$, hence $\mu_i^{cs}=0$ and from Theorem \ref{theorem3} we conclude that $\mu_1$ and $\mu_2$ are quasi-infinitely divisible.\\
The converse is clear.
\end{proof}
In Remark \ref{rem4.11} we characterized quasi-infinite divisibility of $\sum_{i=1}^n p_i N(b_i,a_i)$ as long as $a_1<a_2<\dotso<a_n$. With the aid of Theorem \ref{theorem3}, we can now also consider the case when $a_1\le a_2\le \dotso \le a_n$, provided the $b_i$ satisfy a small restriction:
\begin{example}
Let $\mu=\sum\limits_{i=1}^n p_i \mu_i$ with $b_i,\dotso, b_n\in \bR$ , $0<a_1\le a_2\le \dotso \le a_n$, $0<p_1,\dotso,p_n<1$, $\sum_{i=1}^np_i=1$, $\mu_i\sim N(b_i,a_i)$ and $\hat{\mu}(z)\neq 0$ for every $z\in\bR$. We denote by $J$ the set of indices  for which $b_i=b_1$ and assume that all $b_i, \, i\in J$, lie on a lattice and that
\begin{align*}
\sum\limits_{j\in J} p_j e^{ib_jz}\neq 0
\end{align*}
for $z\in\bR$. Then $\mu$ is quasi-infinitely divisible.
\end{example}
\begin{proof}
We decompose $\mu$ as
\begin{align*}
\mu=\tilde{\mu}\ast(\sum\limits_{i\in \{1,\dotso,n\}\setminus J}p_i \tilde{\mu_i}+\sum\limits_{j\in J}p_j \delta_{b_j})
\end{align*}
with $\tilde{\mu_i}\sim N(b_i,a_i-a_1)$ and $\tilde{\mu}\sim N(0,a_1)$. We conclude from Theorem \ref{theorem3} that $\mu$ is quasi-infinitely divisible.
\end{proof}
\section{Topological Properties of QIDs}
In [\ref{Lindner}, Theorem 4.1] it was shown that the class of infinitely divsible distributions on $\bR$ is dense in the class of all probability distributions with respect to weak convergence. Since distributions exists that are not quasi-infinitely divisible, the class of quasi-infinitely divisible distributions can not be closed, unlike the class of infinitely divisible distributions. In this section we show that the class of infinitely divisible distributions is neither open. However, it is path-connected.\\
Denote by  $\mathcal{P}(\bR)$ the set of all probability measures. Denote by $\pi:\mathcal{P}(\bR)\times\mathcal{P}(\bR)\to [0,\infty)$ the Prokhorov metric on $\mathcal{P}(\bR)$. Then it is known that $(\mathcal{P}(\bR),\pi)$ is a complete metric space, and that the topology defined by the weak convergence is the same as for $\pi$, i.e. for a sequence $(\mu_n)_{n\in\bN}\subset\mathcal{P}(\bR)$ and $\mu\in \mathcal{P}(\bR)$ weak convergence of $\mu_n$ to $\mu$ is equivalent to $\pi(\mu_n,\mu)\to 0$ as $n\to \infty$, see [\ref{Billingsley}, Theorem 6.8, p. 73]. Now we can show: 
\begin{proposition}
The set $QID(\bR)$ of all quasi-infinitely divisible distribution on $\bR$ is not open in the space $(\mathcal{P}(\bR),\pi)$. Moreover, $\mathcal{P}(\bR)\setminus QID(\bR)$ is dense in $(\mathcal{P}(\bR),\pi)$.
\end{proposition}
\begin{proof}
Let $\mu,\nu\in \mathcal{P}(\bR)$ be such that the characteristic function $\hat{\nu}$ has zeroes on $\bR$. We define the sequence of measures
\begin{align*}
\mu_n(dx)=\mu(dx)\ast \nu(n\,dx)
\end{align*}
for every $n\in \bN$. Then $\mu_n$ is clearly not quasi-infinitely divisible as its characteristic function has zeroes on $\bR$, since
\begin{align*}
\hat{\mu_n}(z)=\hat{\mu}(z)\hat{\nu}(\frac{z}{n}).
\end{align*}
Moreover, $\mu_n\stackrel{d}{\to}\mu$. This shows that $\mathcal{P}(\bR)\setminus QID(\bR)$ is dense in $(\mathcal{P}(\bR),\pi)$. In particular, $QID(\bR)$ cannot be open.
\end{proof}
Now we show that $QID(\bR)$ is path-connected. Recall that for a metric space $(X,d)$, a subset $Y\subseteq X$ is called path-connected if for every $x,y\in Y$ there exists a continuous function $p:[0,1]\to Y$ such that $p(0)=x$ and $p(1)=y$.
\begin{proposition}
The space of quasi-infinitely divisible distributions is path-connected, especially connected.
\end{proposition}
\begin{proof}
Let $\mu_1$ and $\mu_2$ be two quasi-infinitely divisible distributions. Then it holds that
\begin{align*}
\mu_t(dx):= \mu_1(\frac{1}{t}dx)\ast \mu_2(\frac{1}{1-t}dx)
\end{align*}
is also quasi-infinitely divisble for $t\in (0,1)$. This holds because
\begin{align*}
p_t(dx):=\mu_1(\frac{1}{t}dx)
\end{align*}
has the characteristic function
\begin{align*}
\hat{p_t}(z):=\hat{\mu_1}(zt)=\exp\left( -\frac{1}{2}a t^2z^2+i\gamma t z+\int\limits_{\bR}\left(e^{izx}-1-izx\one_{[-1,1]}\right)\nu(\frac{1}{t}dx) \right)
\end{align*}
where $\mu_1$ is quasi-infinitely divisible with characteristic triplet $(a,\gamma,\nu)$. Similarly
\begin{align*}
\mu_2(\frac{1}{1-t}dx)
\end{align*}
is also quasi-infinitely divisible. We conclude that $\mu_t(dx)$ is quasi-infinitely divisible for every $t\in(0,1)$. Moreover
$p:[0,1]\to \mathcal{P}(\bR)$ with $p(0)=\mu_1$, $p(1)=\mu_2$ and $p(t)=\mu_t$ is continuous, because it holds $\hat{\mu_t}(z)\to \hat{\mu_{t_0}}(z)$ for $t\to t_0$ for every $z\in\bR$. Hence $QID(\bR)$ is path connected. Finally, observe that path-connectness implies connectness, see [\ref{Armstrong}, Theorem 3.29, p. 61].
\end{proof}
\section*{Acknowledgement:} The author would like to thank Alexander Lindner for introducing him to this topic, giving him the opportunity to work on it and for many interesting and fruitful discussions.

\vspace{1cm}
David Berger\\
 Ulm University, Institute of Mathematical Finance, Helmholtzstra{\ss}e 18, 89081 Ulm,
Germany\\
email: david.berger@uni-ulm.de
\end{document}